 \documentclass[a4paper,12pt]{article}
 \usepackage[T1]{fontenc}
\usepackage[utf8]{inputenc}
\usepackage{authblk}

\usepackage{graphicx,amsmath,amssymb,euscript,amsfonts}
\usepackage{textcomp}
\graphicspath{{images/}}
\usepackage[english]{babel}
\usepackage{amsthm,amssymb,amsmath}
\usepackage{epsfig}
\usepackage{color}
\usepackage[usenames,dvipsnames]{pstricks}
\usepackage{tikz}
\input epsf

\newtheorem{theorem}{Theorem}
\newtheorem{Conjecture}[theorem]{Conjecture}
\newtheorem{proposition}[theorem]{Proposition}
\newtheorem{corollary}[theorem]{Corollary}

\newtheorem{lemma}[theorem]{Lemma}

\newtheorem{Remark}[theorem]{Remark}

\begin{document}
\title{Some results on $\sigma_{t}$-irregularity}

\markright{Abbreviated Article Title}

\author [a]{Slobodan Filipovski \thanks{Corresponding author; {slobodan.filipovski@famnit.upr.si}}}

\affil[a] {\small \it FAMNIT, University of Primorska, Koper, Slovenia}

\author [b]{Darko Dimitrov}

\affil[b] {\small \it Faculty of Information Studies in Novo Mesto, Slovenia}

\author [c]{ Martin Knor}

\affil[c] {\small \it Slovak University of Technology in Bratislava, Slovakia}

\author [a,b,d, e]{ Riste Škrekovski}
\affil[d] {\small \it Faculty of Mathematics and Physics, University of Ljubljana, Slovenia}
\affil[e] {\small \it Rudolfovo - Science and Technology Centre Novo Mesto, Slovenia}

 \date{}
\maketitle

\begin{abstract}% Dimitrov and Stevanović in \cite{ds} defined a new irregularity measure, called $\sigma_{t}$-irregularity index, which represents a natural extension of the $\sigma$-irregularity 
%index. 
The $\sigma_{t}$-irregularity (or sigma total index) is a graph invariant which is defined as $\sigma_{t}(G)=\sum_{\{u,v\}\subseteq V(G)}(d(u)-d(v))^{2},$ where $d(z)$ denotes the degree of $z$. 
This irregularity measure was proposed by R\' {e}ti [Appl. Math. Comput. 344-345 (2019) 107-115], and recently rediscovered by Dimitrov and Stevanović [Appl. Math. Comput. 441 (2023) 127709].
In this paper we remark that $\sigma_{t}(G)=n^{2}\cdot\rm{Var}(\emph{G})$, where $\rm{Var}(\emph{G})$  is the degree variance of the graph. Based on this observation, we characterize irregular 
graphs with maximum $\sigma_{t}$-irregularity.
We show that among all connected graphs on $n$ vertices, the split graphs $S_{\lceil\frac{n}{4}\rceil, \lfloor\frac{3n}{4}\rfloor }$ and $S_{\lfloor\frac{n}{4}\rfloor, \lceil\frac{3n}{4}\rceil }$ 
have the maximum $\sigma_{t}$-irregularity, and among all complete bipartite graphs on $n$ vertices, either the complete bipartite graph $K_{\lfloor\frac{n}{4}(2-\sqrt{2})\rfloor, 
\lceil\frac{n}{4}(2+\sqrt{2})\rceil }$ or $K_{\lceil\frac{n}{4}(2-\sqrt{2})\rceil, \lfloor\frac{n}{4}(2+\sqrt{2})\rfloor }$ has the maximum sigma total index.
Moreover, various upper and lower bounds for $\sigma_{t}$-irregularity are provided; in this direction we give a relation between the graph energy $\mathcal{E}(G)$ and sigma total index 
$\sigma_{t}(G)$ and give another proof of two results by Dimitrov and Stevanović.
Applying Fiedler's characterization of the largest and the second smallest Laplacian eigenvalue of the graph, we also establish new relationships between $\sigma_{t}$ and $\sigma$. 
We conclude the paper with two conjectures.

\end{abstract}

\maketitle

%\baselineskip=0.23in
%\begin{document}

%\footnotetext[1] {Corresponding author.}

%\footnotetext{{ {\it E-mail addresses: akbar.jahanbani92@gmail.com}   \tt } (A. Jahanbani)
%}

%\baselineskip=0.30in

%Let $$\sigma_{t}=\sum_{1\leq i<j \leq n} (d_{i}-d_{j})^{2}.$$
%wo useful lemmas:

\section{Introduction}
Let $G$ be an undirected graph with $n$ vertices and $m$ edges without loops and multiple edges. The degree of a vertex $v$, denoted by $d(v)$, is the number of vertices connected to $v$.
Since $\sum_{v\in V(G)}d(v)=2m,$ the average of vertex degrees is $\overline{d}=\overline{d}(G)=\frac{2m}{n}.$
A graph is $k$-regular if every degree is equal to $k$. Otherwise, the graph is said to be an
irregular graph. Once we know the average degree of a graph, it is possible to compute
more complex measures of the heterogeneity in connectivity across vertices (e.g., the extent
to which there is a very big spread between well-connected and not so well-connected
vertices in the graph) beyond the simpler measures of range such as the difference between
the maximum degree $\Delta$ and the minimum degree $\delta$.
One such measure was proposed by Tom Snijders in \cite{snijders}, called the \emph{degree variance} of the graph.
This vertex-based measure is defined as the average squared deviation between
the degree of each vertex and the average degree:
\begin{equation}\label{var}
\textrm{Var(\emph{G})}=\frac{\sum_{u\in V(G)}(d(u)-\overline{d})^{2}}{n}=\frac{\sum_{u\in V(G)}(d(u)-\frac{2m}{n})^{2}}{n}.
\end{equation}
Note that the degree variance of a regular graph is always equal to zero. The novel irregularity $\rm{IRV}(G)$ which is considered as a modified version of $\rm{Var}(\emph{G})$ index was defined by R\' 
{e}ti \cite{reti} as
\begin{equation}
\rm{IRV}(\emph{G})=\emph{n}^{2}\cdot \rm{Var}(\emph{G}).
\end{equation}
A general edge-additive index is defined as the sum, over all edges, of edge effects. These
effects can be of various types, but the most common ones are defined in terms of some
property of the end-vertices of the considered edge. In many cases the edge contribution
represents how similar its end-vertices are.
The Albertson irregularity of a connected graph $G$, introduced by Albertson \cite{alb} in 1997, was defined by
$$
\textrm{irr(\emph{G})}=\sum_{uv\in E(G)}|d_{G}(u)-d_{G}(v)|,
$$
where $|d_{G}(u)-d_{G}(v)|$ is called the imbalance of an edge $uv$.
This index has been of interest to mathematicians, chemists and scientists from related
fields due to the fact that the Albertson irregularity plays a major role in irregularity
measures of graphs predicting the biological activities and properties of
chemical compounds in the QSAR/QSPR modeling and the quantitative characterization
of network heterogeneity.
Due to their simple computation, the degree variance $\rm{Var}(\emph{G})$
and the Albertson irregularity $\rm{irr}(\emph{G})$ belong to the family of the widely used irregularity indices,
%Two consecutive sentences start with Due. So I removed the second one
%Due to its simplicity, this irregularity measure attracted a
%lot of attention of the researchers
see \cite{abdo, abdo1, alb, chen, reti}.

One natural variation of the Albertson irregularity is the irregularity
measure $\sigma(G)$ \cite{gutman1}, defined as
\begin{equation}
\label{sigma}
\sigma(G)=\sum_{uv\in E(G)}(d(u)-d(v))^{2}.
\end{equation}
Simple calculations show that $\sigma(G)=F(G)-2M_{2}(G)$, where $F(G)=\sum_{u\in V(G)} d(u)^{3}$ is the forgotten topological index \cite{fugu} and $M_{2}(G)=\sum_{uv\in E(G)}d(u)d(v)$ is the second Zagreb index of 
$G.$ For the sake of completeness we also recall the first Zagreb index $M_{1}(G)=\sum_{u \in V(G)} d(u)^{2}.$
More results on $\sigma$-irregularity can be found in \cite{abdo2, abdo3}.
Since graphs with the same degree sequence do not necessarily have the same $\sigma$-irregularity, Dimitrov and Stevanović \cite{ds} introduced a modification of $\sigma$-irregularity. This 
modification is called the \emph{$\sigma_t$-irregularity} and is defined as
\begin{equation}
\label{total}
\sigma_{t}(G)=\sum_{\{u,v\}\subseteq V(G)}(d(u)-d(v))^{2}.
\end{equation}

We recall three results on $\sigma_{t}$ from \cite{ds}.

\begin{proposition}
\label{eden}
Let $G$ be a simple connected graph with $n$ vertices and $m$ edges. Then
$$
\sigma_{t}(G)=nM_{1}(G)-4m^{2}.
$$
\end{proposition}

\begin{proposition}
\label{dva}
Among all trees on $n$ vertices, the star has the maximum $\sigma_{t}$-irregularity.
\end{proposition}

\begin{proposition}
\label{tri}
Among all trees on $n$ vertices, the path graph has the minimal $\sigma_{t}$-irregularity.
\end{proposition}

%Motivated by the results in \cite{ds}
In this paper we present some further results on $\sigma_{t}$-irregularity.
First, we observe that  $\sigma_{t}(G)=n^{2}\cdot \rm{Var}(\emph{G})=\rm{IRV}(\emph{G}),$
which means that the sigma total index defined in \cite{ds} and the irregularity measure $\rm{IRV}(\emph{G})$ defined in \cite{reti} represent the same quantity, but the equality between them was 
not discovered before. By using this equality and known properties of the variance, we characterize the irregular graphs with maximum and minimal $\sigma_{t}$-irregularity.
Moreover, in Section 5  we give new relations between the $\sigma$- and $\sigma_{t}$-irregularity.

\section{Preliminary Results}

In this section we establish several lemmas, which are necessary for our main results.

\begin{lemma}
\label{četiri}
Let $G$ be a graph on $n$ vertices and $m$ edges. Then
\begin{equation}
\label{lema}
\sigma_{t}(G)=n^{2}\cdot \rm{Var}(\textit{G}),
\end{equation}
where $\rm{Var}(\textit{G})$ is the degree variance of the graph.
\end{lemma}

\begin{proof} Let $d_{1}\geq d_{2}\geq \cdots \geq d_{n}$ be the degree sequence of $G$. We have
$$
\sigma_{t}(G)=\sum_{1\leq i < j \leq n} (d_{i}-d_{j})^{2}=n \sum_{i=1}^{n}d_{i}^{2}-(d_{1}+\cdots+d_{n})^{2}$$
$$=n\left(\sum_{i=1}^{n}d_{i}^{2}-n\left(\frac{d_{1}+\cdots+d_{n}}{n}\right)^{2}\right)=n\left(\sum_{i=1}^{n}d_{i}^{2}-n\cdot\overline{d}^{2}\right)=n\left(\sum_{i=1}^{n}d_{i}^{2}-2n\overline{d}^{2}+n\overline{d}^{2}\right)$$
$$=n\left(\sum_{i=1}^{n}d_{i}^{2}-2\overline{d}\cdot (d_{1}+\cdots+d_{n})+n\overline{d}^{2}\right)=n\left((d_{1}-\overline{d})^{2}+\cdots+(d_{n}-\overline{d})^{2}\right)=n^{2}\cdot 
\rm{Var}(\emph{G}).
$$
\end{proof}

\begin{lemma}\label{pet}
Let $(A)=\{a_i\}_{i=1}^n$ be a sequence of real numbers, such that $a_{1}\geq a_{2}\geq \cdots \geq a_{n}$.
For fixed $i$ and $j$, $i<j,$ we form a new sequence $(B)$, based on $(A)$,  where the $i$-th term of $(A)$ is increased by one and the $j$-th term is decreased by one.
%, that is, $(B): a_{1}\geq a_{2}\geq \cdots \geq a_{i}+1 \geq \cdots \geq a_{j}-1 \geq \cdots \geq a_{n}.$
Then
$$
\rm{Var}(\textit{A})<\rm{Var}(\textit{B}).
$$
\end{lemma}

\begin{proof}
First, notice that both sequences have the same mean $\overline{a}=\frac{a_{1}+\cdots+a_{n}}{n}$.
Hence
$$
n\cdot \rm{Var}(\emph{B})-\textit{n}\cdot 
\rm{Var}(\emph{A})=(\emph{a}_{\textit{i}}+1-\overline{\emph{a}})^{2}+(\emph{a}_{\textit{j}}-1-\overline{\emph{a}})^{2}-(\emph{a}_{\textit{i}}-\overline{\emph{a}})^{2}-(\emph{a}_{\textit{j}}-\overline{\emph{a}})^{2}=2(\emph{a}_{\textit{i}}-\emph{a}_{\textit{j}})+2>0.
$$
\end{proof}

In order to simplify the characterization of irregular graphs with maximum $\sigma_{t}$-irregularity we use the following lemma.

\begin{lemma}\label{šest}
For any positive integer $n\ge 2$ it holds
\begin{itemize}
\item [a)]
$$
\left\lceil \frac{1}{8}(5n-2 - \sqrt{9n^{2}-4n+4})\right\rceil =\left\lceil \frac{n}{4}\right\rceil
$$
\item [b)]
\begin{equation}
\left\lfloor \frac{1}{8}(5n-2 - \sqrt{9n^{2}-4n+4})\right\rfloor =
\left\{
\begin{array}{ll} \lfloor\frac{n}{4}\rfloor -1, & \mbox{if $n$ is divisible by 4}; \\
\lfloor\frac{n}{4}\rfloor, & \mbox{ otherwise}. \\
\end{array}
\right.
\end{equation}
\end{itemize}

\end{lemma}

\begin{proof} It suffices to prove that
\begin{equation}
\label{ineq}
\frac{n}{4}-\frac{1}{4}< \frac{1}{8}(5n-2 - \sqrt{9n^{2}-4n+4}) <\frac{n}{4}.
\end{equation}

The reason is that (\ref{ineq}) implies
$\left\lceil\frac{n-1}4\right\rceil\le
\left\lceil\frac{1}{8}(5n-2 - \sqrt{9n^{2}-4n+4})\right\rceil
\le\left\lceil\frac n4\right\rceil$,
which gives the equation
$\left\lceil \frac{1}{8}(5n-2 - \sqrt{9n^{2}-4n+4})\right\rceil 
=\left\lceil\frac n4\right\rceil$
for every $n$ except $n\equiv 1\pmod 4$.
However, if $n\equiv 1\pmod 4$ then 
$$
\left\lceil\frac{n-1}4\right\rceil=\frac n4-\frac 14<
\frac{1}{8}(5n-2 - \sqrt{9n^{2}-4n+4})
<\left\lceil\frac n4\right\rceil
$$
and consequently $\left\lceil \frac{1}{8}(5n-2 - \sqrt{9n^{2}-4n+4})\right\rceil 
=\left\lceil\frac n4\right\rceil$ also in this case.

On the other hand, (\ref{ineq}) implies
$\left\lfloor\frac{n-1}4\right\rfloor\le
\left\lfloor\frac{1}{8}(5n-2 - \sqrt{9n^{2}-4n+4})\right\rfloor
\le\left\lfloor\frac n4\right\rfloor$,
which gives the equation
$\left\lfloor\frac{1}{8}(5n-2 - \sqrt{9n^{2}-4n+4})\right\rfloor
=\left\lfloor\frac n4\right\rfloor$ for every $n$ except $n\equiv 0\pmod 4$.
If $n\equiv 0\pmod 4$ then
$$
\frac{n-1}4<
\frac{1}{8}(5n-2 - \sqrt{9n^{2}-4n+4})
<\frac n4=\left\lfloor\frac n4\right\rfloor
$$
and so $\left\lfloor \frac{1}{8}(5n-2 - \sqrt{9n^{2}-4n+4})\right\rfloor 
=\left\lfloor\frac n4\right\rfloor-1$.

Observe that the right-hand side of (\ref{ineq}) is equivalent to $3n-2 <\sqrt{9n^{2}-4n+4}$ which holds for every $n> 0$.
Similarly, the left-hand side of (\ref{ineq}) is equivalent to $3n>\sqrt{9n^{2}-4n+4}$, which is true for $n>1$.
Consequently, (\ref{ineq}) holds for $n\ge 2$.
\end{proof}

%Similarly like Lemma 3, the next lemma will be applied in the classification of the bipartite graphs with maximal sigma total index.

%\begin{lemma} Let $n\geq 12$ be a natural number. Then
%\begin{itemize}
%\item [a)]  $\lfloor \frac{5n^{2}-\sqrt{n^{4}-12n^{3}}}{3(8n-4)}\rfloor=\lfloor\frac{n}{6}\rfloor$ or $\lceil\frac{n}{6}\rceil.$

%\end{itemize}
%\end{lemma}
%\begin{proof}
%\begin{itemize}
%\item [a)] It suffices to show that $\frac{n}{6}\leq \frac{5n^{2}-\sqrt{n^{4}-12n^{3}}}{3(8n-4)}\leq \frac{n}{6}+1.$  The left hand side of the inequality is equivalent to $16n+4\geq 0$, which is 
%true. The right side of the inequality is equivalent to $n^{2}-22n+12\leq n\sqrt{n^{2}-12n},$ which is true for $n\geq 12.$
    
%\item [b)] In this case is suffices to show that $\frac{n}{4}-1\leq \frac{5n^{2}+\sqrt{n^{4}-12n^{3}}}{3(8n-4)}\leq \frac{n}{4}.$  The right hand side of the inequality is equivalent to $n-3\geq 
%\sqrt{n^{2}-12n}$, which is true. The left side of the inequality is equivalent to $n^{2}-27n+12\leq n\sqrt{n^{2}-12n}$, which is true for $n\geq 12.$

%\end{itemize}
%\end{proof}
One of the upper bounds for $\sigma_{t}(G)$ is based on the following improvement between arithmetic and geometric means, published in \cite{filipovski}.

\begin{lemma}\label{sedum}If $x$ and $y$ are strictly positive numbers, then
$$\sqrt{\frac{x}{y}}+\sqrt{\frac{y}{x}}\geq 2+\frac{(x-y)^{2}}{2(x^{2}+y^{2})}.$$
\end{lemma}

Since the sigma total index is equal to $n^{2}\cdot \rm{Var}(\emph{G})$, we use the following result by Bhatia and Davis \cite{bhatia}.

\begin{lemma} \label{osum} Let $a_{1}, a_{2}, \ldots, a_{n}$ be real numbers. Assume $M\geq a_{j}$ and $m\leq a_{j}$ for all $j$, and set $\overline{a}$ for the average: $\overline{a}=\frac{1}{n} \sum_{j=1}^{n} a_{j}.$ Then
$$\frac{1}{n} \sum_{j=1}^{n}(a_{j}-\overline{a})^{2} \leq (M-\overline{a})(\overline{a}-m).$$
Equality holds if and only if every $a_{j}$ is equal either to $M$ or to $m$.
\end{lemma}

\section{Graphs with maximum $\sigma_{t}$-irregularity}

For the sake of Theorem \ref{prvaa}, we recall the definition of split graphs.
The \emph{split graph} is a graph $S_{a,b}$ whose vertices can be partitioned into a clique of size $a$ and an independent set of vertices of size $b$, and every vertex of the independent set is 
connected to all vertices of the clique.

\begin{theorem}\label{prvaa} For a connected graph $\Gamma$ on $n\in \mathbb{N}, n\geq 3$ vertices it holds
$$
\sigma_{t}(\Gamma)\leq \left\{
\begin{array}{lc} \lceil \frac{n}{4}\rceil \cdot \lfloor \frac{3n}{4}\rfloor \left(n-1-\lceil \frac{n}{4}\rceil \right)^{2}  , \mbox { if\;} \; n \equiv 0 \bmod 4 \mbox { or\; } n \equiv 3 \bmod 4; 
\\
\lfloor \frac{n}{4}\rfloor \cdot \lceil \frac{3n}{4}\rceil \left(n-1-\lfloor \frac{n}{4}\rfloor \right)^{2} , \mbox { if\;}\; n \equiv 1 \bmod 4 \mbox { or\; } n \equiv 2 \bmod 4. \\
\end{array}
\right.
$$

If $n \equiv 0 \bmod 4$ or $ n \equiv 3 \bmod 4$ the equality holds for the split graph $S_{\lceil\frac{n}{4}\rceil, \lfloor\frac{3n}{4}\rfloor },$ while if $n \equiv 1 \bmod 4$ or $ n \equiv 2 
\bmod 4$ the equality holds for the split graph $S_{\lfloor\frac{n}{4}\rfloor, \lceil\frac{3n}{4}\rceil }.$
\end{theorem}

\begin{proof} Let $G$ be a connected graph on $n\geq 3$ vertices with the maximum $\sigma_{t}$-irregularity among all connected graphs on $n$ vertices. Let $\Delta=d_{1}\geq d_{2}\geq \cdots \geq d_{n}=\delta$ be the degree sequence of $G$ and let $\overline{d}$ be its average degree. It holds $\delta\leq \overline{d} \leq \Delta.$ 
If $d_{1}\neq n-1$, then $v_{1}$ is not connected to at least one vertex $u$.
Let $v_{1}\cdots v_{j}u$ be a shortest path between $v_{1}$ and  $u$.

We form a new graph $G'$ by adding an edge between $v_{1}$ and $u$ and removing the edge $v_{j}u$.
Since $v_{1}$ and $u$ are not adjacent, it follows that $v_{1}\neq v_{j}$.
In this case the graph $G'$ remains connected.
%The degree sequence of $G^{'}$ is: $d_{1}+1\geq \cdots \geq d_{j}-1\geq\cdots\geq  d_{i}\geq \cdots \geq d_{n}$ or $d_{1}+1\geq \cdots \geq d_{i}\geq \cdots \geq d_{j}-1\geq \cdots \geq d_{n}.$ 
%Or d_n\ge d_j-1. Too many possibilities.
By Lemma \ref{četiri} and Lemma \ref{pet}, we get $\sigma_{t}(G)=n^{2}\cdot \rm{Var}(\emph{G})<\emph{n}^{2}\cdot \rm{Var}(\emph{G}')=\sigma_\emph{t}(\emph{G}')$.

Since by increasing $d_{1}$ we increase the $\sigma_{t}$-irregularity, without loss of generality we may assume that $d_{1}=n-1.$ Moreover, let $x$  be the number of vertices of maximum degree $n-1$.
Since the remaining $n-x$ vertices are connected to each of these $x$ vertices, we obtain that the smallest possible degree of the graph $G$ is at least $x$. This yields $x\leq d_{i} \leq n-1.$ From Lemma 8 we get that 
\begin{equation}\label{maxx} 
\sigma_{t}(\Gamma)\leq \sigma_{t}(G)=n^{2}\cdot \rm{Var}(\emph{G})\leq \textit{n}^{2}(\textit{n}-1-\overline{\textit{d}})(\overline{\textit{d}}-\textit{x}).
\end{equation}
The equality in (\ref{maxx}) holds if every vertex degree $d_{i}$ is equal either to $n-1$ or $x$. Consequently we assume that $G$ contains $x$ vertices of degree $n-1$ and $n-x$ vertices of degree $x$. Hence, $G$ is a split graph $S_{x, n-x}$ where its vertices are partitioned into a clique of size $x$ and an independent set of size $n-x$.
Since the vertices in the clique are of the same degree, the pairs of these vertices do not contribute to the sigma total index.
The same observation holds for the vertices in the independent set. Thus, we count the number of edges between both sets, which is equal to $x(n-x)$.
We get
\begin{equation}\label{glavno}
\sigma_{t}(\Gamma)\leq \sigma_{t}(G)=\sigma_{t}(S_{x,n-x})=x(n-x)(n-1-x)^{2}.
\end{equation}
From (\ref{maxx}), we note that the largest value of $\sigma_{t}(G)$ is equal to
$$n^{2}(n-1-\overline{d})(\overline{d}-x)$$
$$=n^{2}\left(n-1- \frac{x(n-1)+(n-x)x}{n}\right)\left(\frac{x(n-1)+(n-x)x}{n}-x\right)$$
$$=x(n-x)(n-1-x)^{2},$$
which matches with the $\sigma_{t}(S_{x,n-x}).$

%It implies
%\begin{equation}\label{var}
%(d_{1}-\overline{d})^{2}+\cdots+(d_{n}-\overline{d})^{2}=x(n-1-\overline{d})^{2}+(n-x)(x-\overline{d})^{2}
%\end{equation}
%where $\overline{d}=\frac{2m}{n}=\frac{x(n-1)+(n-x)x}{n}.$ Now (\ref{var}) becomes
%$$(d_{1}-\overline{d})^{2}+\cdots+(d_{n}-\overline{d})^{2}=x\left(n-1-\frac{x(n-1)+(n-x)x}{n}\right)^{2}+(n-x)\left(x-\frac{x(n-1)+(n-x)x}{n}\right)^{2}=$$
%$$=x(n-1)^{2}+(n-x)x^{2}-2\left(\frac{x(n-1)+(n-x)x}{n}\right)x(2n-x-1)+\frac{(x(n-1)+(n-x)x)^{2}}{n}=f(x).$$
In order to maximize the function $\sigma_{t}(S_{x,n-x}),$ we find its critical points. We get
$$\frac{\partial \sigma_{t}}{\partial x}= -4x^{3}+3x^{2}(3n-2)-2x(3n^{2}-4n+1)+n^{3}-2n^{2}+n.$$
From $\frac{\partial \sigma_{t}}{\partial x}=0$ we get $x=n-1$ or $x=\frac{1}{8}(5n-2 \pm \sqrt{9n^{2}-4n+4}).$
It is easy to check that $n>\frac{1}{8}(5n-2 + \sqrt{9n^{2}-4n+4})\geq n-1$.
Using the first derivative we conclude that $\sigma_{t}$ is increasing on $\big(0, \frac{1}{8}(5n-2 - \sqrt{9n^{2}-4n+4})\big)$ and on $\big(n-1, \frac{1}{8}(5n-2 + \sqrt{9n^{2}-4n+4})\big),$ and 
decreasing on
the intervals $\big(\frac{1}{8}(5n-2 - \sqrt{9n^{2}-4n+4}), n-1\big)$ and $\big(\frac{1}{8}(5n-2 + \sqrt{9n^{2}-4n+4}), n\big)$.
If $x=n-1$, then $G$ is a complete graph, and then $\sigma_{t}(G)=0.$
Thus the largest possible value for $\sigma_{t}$ is achieved
for the split graphs $S_{x, n-x}$, where $x=\left\lfloor \frac{1}{8}(5n-2 - \sqrt{9n^{2}-4n+4})\right\rfloor$ or $x=\left\lceil \frac{1}{8}(5n-2 - \sqrt{9n^{2}-4n+4})\right\rceil$, and 
$\sigma_{t}(G)=x(n-x)(n-1-x)^{2}$.

In the following we compare both possible values for $\sigma_{t}$, choosing the larger one.
Let $G_{1}$ and $G_{2}$ be two split graphs that contain $x_{1}=\left\lfloor \frac{1}{8}(5n-2 - \sqrt{9n^{2}-4n+4})\right\rfloor$  and $x_{2}=\left\lceil \frac{1}{8}(5n-2 - 
\sqrt{9n^{2}-4n+4})\right\rceil$ vertices of degree $n-1$, respectively.
We apply Lemma \ref{šest} and consider the following cases:

\begin{itemize}
\item[1.]
Let $n=4k.$ Then $x_{1}=\lfloor \frac{n}{4}\rfloor -1=k-1$ and $x_{2}=\lceil \frac{n}{4}\rceil=k.$ We get
$$\sigma_{t}(G_{1})-\sigma_{t}(G_{2})= (k-1)\cdot (4k-k+1)\cdot (4k-1-(k-1))^{2}- k\cdot (4k-k)(4k-1-k)^{2}=-12k^{2}<0.$$
In this case we have $\sigma_{t}(G)\leq \sigma_{t}(G_{2})=\lceil \frac{n}{4}\rceil \cdot \lfloor \frac{3n}{4}\rfloor \left(n-1-\lceil \frac{n}{4}\rceil \right)^{2}.$

\item [2.]
Let $n=4k+1.$ Then $x_{1}=\lfloor \frac{n}{4}\rfloor=k$ and $x_{2}=\lceil \frac{n}{4}\rceil=k+1.$ We get
$$\sigma_{t}(G_{1})-\sigma_{t}(G_{2})=k\cdot (4k+1-k)\cdot (4k-k)^{2}-(k+1)\cdot (4k+1-k-1)(4k-k-1)^{2}=$$
$$= k\cdot (3k+1)\cdot 9k^{2}-(k+1)\cdot 3k \cdot (3k-1)^{2}=15k^{2}-3k>0.$$
In this case we have $\sigma_{t}(G)\leq \sigma_{t}(G_{1})=\lfloor \frac{n}{4}\rfloor \cdot \lceil \frac{3n}{4}\rceil \left(n-1-\lfloor \frac{n}{4}\rfloor \right)^{2}.$

\item [3.]
Let $n=4k+2.$ Then $x_{1}=\lfloor \frac{n}{4}\rfloor=k$ and $x_{2}=\lceil \frac{n}{4}\rceil=k+1.$ We get
$$\sigma_{t}(G_{1})-\sigma_{t}(G_{2})= k\cdot (4k+2-k)\cdot (4k+1-k)^{2}- (k+1)\cdot (4k+2-k-1)(4k+1-k-1)^{2}=$$
$$= k\cdot (3k+2)\cdot (3k+1)^{2}-(k+1)\cdot (3k+1) \cdot (3k)^{2}=6k^{2}+2k>0.$$
In this case we have $\sigma_{t}(G)\leq \sigma_{t}(G_{1})=\lfloor \frac{n}{4}\rfloor \cdot \lceil \frac{3n}{4}\rceil \left(n-1-\lfloor \frac{n}{4}\rfloor \right)^{2}.$

\item [4.]
Let $n=4k+3.$ Then $x_{1}=\lfloor \frac{n}{4}\rfloor=k$ and $x_{2}=\lceil \frac{n}{4}\rceil=k+1.$ We get
$$\sigma_{t}(G_{1})-\sigma_{t}(G_{2})= k\cdot (4k+3-k)\cdot (4k+2-k)^{2}-(k+1)\cdot (4k+3-k-1)(4k+2-k-1)^{2}=$$
$$=k\cdot (3k+3)\cdot (3k+2)^{2}- (k+1)\cdot (3k+2) \cdot (3k+1)^{2}=-3k^{2}-5k-2<0.$$
Similarly as in the first case, we have $\sigma_{t}(G)\leq \sigma_{t}(G_{2})=\lceil \frac{n}{4}\rceil \cdot \lfloor \frac{3n}{4}\rfloor \left(n-1-\lceil \frac{n}{4}\rceil \right)^{2}.$
\end{itemize}
This case analysis completes the proof.

%Comparing both values we get that the largest sigma total index is obtained for $x=\lceil\frac{n}{4}\rceil.$
%Therefore, the graph $G$ contains $\lceil \frac{n}{4}\rceil$ vertices of degree $n-1$, and $\lfloor \frac{3n}{4}\rfloor $ vertices of degree $\lceil \frac{n}{4}\rceil$.
%Hence
%$$\sigma_{t}(G)\leq n \cdot \lceil \frac{n}{4}\rceil \cdot \lfloor \frac{3n}{4}\rfloor \left(n-1-\lceil \frac{n}{4}\rceil \right)^{2}.$$
%The equality holds for the split graphs $S_{\lceil\frac{n}{4}\rceil, \lfloor\frac{3n}{4}\rfloor }.$

\end{proof}

\subsection{Graphs without triangles}

In this subsection we consider triangle-free graphs. First, we give an upper bound for $\sigma_{t}(G)$ in terms of $m$ and $n$.

\begin{proposition}\label{vtoraa} Let $G$ be a triangle-free graph on $n$ vertices and with $m$ edges.
Then
$$
\sigma_{t}(G)\leq m(n^{2}-4m).
$$
The equality holds if and only if $G$ is a complete bipartite graph.
\end{proposition}

\begin{proof}
The proof is straightforward. For triangle-free graphs it holds $M_{1}(G)\leq mn,$ and the equality holds if and only if $G$ is a complete bipartite graph, see \cite{zhou}.
 This inequality and Proposition \ref{eden} yield
$$\sigma_{t}(G)=nM_{1}(G)-4m^{2}\leq mn^2-4m^{2}=m(n^{2}-4m).$$
\end{proof}

For $2\leq n\leq 9$, we find that the complete bipartite graphs $K_{1, n-1}$ have the maximum sigma total index in triangle-free graphs.
For $n=10$, two graphs produce the maximum $\sigma_{t}$, namely $K_{1,9}$ and $K_{2,8}$, while for $n=11$, it is $K_{2,9}$.
We believe that in the class of triangle-free graphs, the maximum sigma total index is achieved by complete bipartite graphs.
The next statement characterizes which complete bipartite graphs yield the maximum sigma total index.
%The next result characterizes the bipartite graphs on $n\geq 12$ vertices with maximal $\sigma_{t}$-irregularity.

\begin{proposition}\label{tretaa}
Among complete bipartite graphs on $n\geq 7$ vertices,  either the complete bipartite graph $K_{\lfloor\frac{n}{4}(2-\sqrt{2})\rfloor, \lceil\frac{n}{4}(2+\sqrt{2})\rceil }$ or 
$K_{\lceil\frac{n}{4}(2-\sqrt{2})\rceil, \lfloor\frac{n}{4}(2+\sqrt{2})\rfloor}$ has the maximum $\sigma_{t}$-irregularity. 
\end{proposition}

\begin{proof}
Let $G$ be a complete bipartite graph with partite sets $V_{1}$ and $V_{2}.$ Also, let us suppose that the cardinalities of the partition sets $V_{1}$ and $V_{2}$ are $n_{1}$ and $n_{2}$, respectively, 
and let $n_{1}\leq n_{2}$.
We have
%
 %us suppose that the bipartite graph $G$ maximizes the sigma total index, and let $\overline{d}$ be its average degree. Also, we suppose that the cardinality of the partition sets $V_{1}$ and 
 %$V_{2}$ are $n_{1}$ and $n_{2}$, respectively. Let $n_{1}\leq n_{2}.$ Therefore $\Delta(G)=d_{1}\leq n_{2}.$  %If $G$ contains a vertex of degree $n-1$, then $n_{1}=1$, $n_{2}=n-1$ and 
 %$\overline{d}=\frac{1\cdot(n-1)+(n-1)\cdot 1}{n}=\frac{2(n-1)}{n}.$ In this case we have
%$$\sigma_{t}(G)=(n-1)(n-2)^{2}.$$
%We apply the same approach from Theorem 1. Let $v_{1}\in V_{1}$ be the vertex with the largest degree among the vertices in $V_{1}.$
%If $d_{1}<n_{2}$, there is a vertex $v_{i}\in V_{2}$ such that $v_{1}$ and $v_{i}$ are not adjacent.
%Let $v_{1}\ldots v_{i-1}v_{i}$ be the shortest path between $v_{1}$ and  $v_{i}.$
%Since this path has an odd length, it is clear that $v_{i-1}\in V_{1}.$  We form a new bipartite graph $G^{'}$ adding an edge between $v_{1}$ and $v_{i}$ and removing the edge $(v_{i-1}, v_{i}).$
%According to Lemma 2, $\sigma_{t}(G^{'})\geq \sigma_{t}(G).$ We repeat this procedure until $d_{1}=n_{2}.$ In order to maximize the index we consider the vertices with maximum and minimum degree. 
%The maximum number of vertices with maximum degree is $n_{1}.$ Thus we suppose that all vertices from $V_{1}$ have degree $n_{2}.$ Hence, all vertices from $V_{2}$ have degree $n_{1}.$ We get
$$
\sigma_{t}(G)=n_{1}n_{2}(n_{1}-n_{2})^{2}=n_{1}(n-n_{1})(2n_{1}-n)^{2}=-4n_{1}^{4}+8nn_{1}^{3}-5n^{2}n_{1}^{2}+n^{3}n_{1}.
$$
We maximize the expression $-4n_{1}^{4}+8nn_{1}^{3}-5n^{2}n_{1}^{2}+n^{3}n_{1}$.
We have
$$
\frac{ \partial \sigma_{t}(G)}{\partial n_{1}}=-16n_{1}^{3}+24nn_{1}^{2}-10n^{2}n_{1}+n^{3}.
$$
From $\frac{\partial \sigma_{t}(G)}{\partial n_{1}}=0$ we get $n_{1}=\frac{n}{2}$, or $n_{1}=\frac{n}{4}(2\pm\sqrt{2})$. 
Note that $\frac{n}{4}(2-\sqrt{2})<\frac{n}{2}<\frac{n}{4}(2+\sqrt{2})<n$.
Based on the first derivative we conclude that $\sigma_{t}$ is increasing on the intervals $\big(0, \frac{n}{4}(2-\sqrt{2})\big)$ and $\big(\frac{n}{2}, \frac{n}{4}(2+\sqrt{2})\big)$ and it is 
decreasing on $\big(\frac{n}{4}(2-\sqrt{2}), \frac{n}{2}\big)$ and $\big(\frac{n}{4}(2+\sqrt{2}), n\big)$.
Since $n_{1}\leq \frac{n}{2},$ the maximum $\sigma_{t}(G)$ is achieved for $n_{1}=\left\lfloor\frac{n}{4}(2-\sqrt{2})\right\rfloor$ or $n_{1}=\left\lceil\frac{n}{4}(2-\sqrt{2})\right\rceil$.

If $n_{1}=\left\lfloor\frac{n}{4}(2-\sqrt{2})\right\rfloor$, then $n_{2}=\left\lceil\frac{n}{4}(2+\sqrt{2})\right\rceil$, and 
\begin{equation}\label{sig}
\sigma_{t}(G)=\left\lfloor\frac{n}{4}(2-\sqrt{2})\right\rfloor \cdot \left\lceil\frac{n}{4}(2+\sqrt{2})\right\rceil \cdot 
\left(\left\lceil\frac{n}{4}(2+\sqrt{2})\right\rceil-\left\lfloor\frac{n}{4}(2-\sqrt{2})\right\rfloor\right)^{2}.
\end{equation}
Similarly, if $n_{1}=\left\lceil\frac{n}{4}(2-\sqrt{2})\right\rceil$, then $n_{2}=\left\lfloor\frac{n}{4}(2+\sqrt{2})\right\rfloor$, and 
\begin{equation}\label{sig1}
\sigma_{t}(G)=\left\lceil\frac{n}{4}(2-\sqrt{2})\right\rceil\cdot \left\lfloor\frac{n}{4}(2+\sqrt{2})\right\rfloor \cdot 
\left(\left\lfloor\frac{n}{4}(2+\sqrt{2})\right\rfloor-\left\lceil\frac{n}{4}(2-\sqrt{2})\right\rceil\right)^{2}.
\end{equation}
\end{proof}

\begin{Remark} Note that the minimum $\sigma_{t}(G)$ is achieved for $n_{1}=\lfloor \frac{n}{2}\rfloor$. If $n$ is even, then $n_{1}=n_{2}=\frac{n}{2}$ and $\sigma_{t}(G)=0.$ If $n$ is odd, then 
$n_{1}=\left\lfloor \frac{n}{2}\right\rfloor$, $n_{2}=\left\lceil \frac{n}{2}\right\rceil$ and 
\begin{equation}\label{prvi}
\sigma_{t}(G)=\left\lfloor \frac{n}{2}\right\rfloor \cdot \left\lceil \frac{n}{2}\right\rceil \cdot \left(\left\lceil \frac{n}{2}\right\rceil-\left\lfloor 
\frac{n}{2}\right\rfloor\right)^{2}=\left\lfloor \frac{n}{2}\right\rfloor \cdot \left\lceil \frac{n}{2}\right\rceil.
\end{equation}
From $n>\left\lceil\frac{n}{4}(2+\sqrt{2})\right\rceil\geq \left\lfloor\frac{n}{4}(2+\sqrt{2})\right\rfloor\geq \left\lfloor \frac{3n}{4}\right\rfloor$ and 
$\left\lfloor\frac{n}{4}(2-\sqrt{2})\right\rfloor\leq \left\lceil\frac{n}{4}(2-\sqrt{2})\right\rceil\leq \left\lceil \frac{n}{4}\right\rceil$ we obtain that the bounds in (\ref{sig}) and 
(\ref{sig1}) are bigger than the bound in (\ref{prvi}).
\end{Remark}

The problem of finding which graph from $K_{\lfloor\frac{n}{4}(2-\sqrt{2})\rfloor, \lceil\frac{n}{4}(2+\sqrt{2})\rceil }$ and $K_{\lceil\frac{n}{4}(2-\sqrt{2})\rceil, 
\lfloor\frac{n}{4}(2+\sqrt{2})\rfloor}$ has a bigger $\sigma_{t}$-irregularity we leave for possible future work.
%Due to the complicated form of (\ref{sig}) and (\ref{sig1}), we do not compare them in this paper. 
%Thus, a fully characterization of the extremal complete bipartite graphs in this paper is missing.
%However, based on a verification, we believe that there may exist an analytical method that could be useful for the solution of this problem.  

 % By Lemma 4 we get
%$\lfloor \frac{5n^{2}-\sqrt{n^{4}-12n^{3}}}{3(8n-4)}\rfloor=\lfloor \frac{n}{6}\rfloor$ or $\lceil \frac{n}{6}\rceil$ and $\lceil \frac{5n^{2}+\sqrt{n^{4}-12n^{3}}}{3(8n-4)}\rceil =\lfloor 
%\frac{n}{4}\rfloor$ or $\lceil \frac{n}{4}\rceil.$ It is easy to check that for $n\geq 12$ it holds

%$$ (n-1)(n-2)^{2}<\lfloor \frac{n}{6}\rfloor \cdot \lceil\frac{5n}{6}\rceil \cdot \left(\lceil\frac{5n}{6}\rceil-\lfloor \frac{n}{6}\rfloor\right)^{2}$$
%and $$ \lfloor \frac{n}{4}\rfloor \cdot \lceil\frac{3n}{4}\rceil \cdot \left(\lceil\frac{3n}{4}\rceil-\lfloor \frac{n}{4}\rfloor\right)^{2}<\lfloor \frac{n}{6}\rfloor \cdot 
%\lceil\frac{5n}{6}\rceil \cdot \left(\lceil\frac{5n}{6}\rceil-\lfloor \frac{n}{6}\rfloor\right)^{2}.$$
%Therefore
%$$\sigma_{t}(G)\leq \lfloor \frac{n}{6}\rfloor \cdot \lceil\frac{5n}{6}\rceil \cdot \left(\lceil\frac{5n}{6}\rceil-\lfloor \frac{n}{6}\rfloor\right)^{2}.$$
%The equality holds for $n_{1}=\lfloor \frac{n}{6}\rfloor$ and $n_{2}=\lceil \frac{5n}{6}\rceil,$ that is, for $G\cong K_{\lfloor \frac{n}{6}\rfloor, \lceil \frac{5n}{6}\rceil}.$

\subsection{Application in graph energy}

Next we give an upper bound for $\sigma_{t}$ in terms of $m,n, \Delta$ and $\delta$.

\begin{proposition}
\label{četvrtaa} 
Let $G$ be a connected graph on $n$ vertices and $m$ edges. Let $\Delta$ and $\delta$ be the maximum and the minimum degree of $G$, respectively. Then
$$\sigma_{t}(G)\leq \frac{4(\sqrt{2mn}-n\sqrt{\delta})(n^{2}\Delta^{2}+4m^{2})}{n\sqrt{\delta}}.$$
\end{proposition}

\begin{proof}
Let $d_{1}\geq \cdots \geq d_{n}$ be the degree sequence of $G$.
Assuming $x=\frac{d_{i}}{2m}$ and $y=\frac{1}{n}$ in Lemma \ref{sedum} we get
\begin{equation}\label{imp}
\frac{d_{i}}{2m}+\frac{1}{n}\geq \left(2+\frac{(nd_{i}-2m)^{2}}{2(n^{2}d_{i}^{2}+4m^{2})}\right)\sqrt{\frac{d_{i}}{2mn}}.
\end{equation}
Now, setting $i=1, \ldots, n$ in (\ref{imp}) and summing up the above inequalities, we have
\begin{equation}\label{en}
\frac{1}{2m} \sum_{i=1}^{n}d_{i}+\sum_{i=1}^{n}\frac{1}{n}\geq \frac{2}{\sqrt{2mn}}\sum_{i=1}^{n}\sqrt{d_{i}}+\sqrt{\frac{\delta}{2mn}}\cdot \frac{1}{2(n^{2}\Delta^{2}+4m^{2})}\cdot 
\sum_{i=1}^{n}(nd_{i}-2m)^{2},
\end{equation}
which is equivalent to
\begin{equation}\label{in}2\geq \frac{2}{\sqrt{2mn}}\sum_{i=1}^{n}\sqrt{d_{i}}+\sqrt{\frac{\delta}{2mn}}\cdot 
\frac{n^{2}}{2(n^{2}\Delta^{2}+4m^{2})}\sum_{i=1}^{n}\left(d_{i}-\frac{2m}{n}\right)^{2}.
\end{equation}
The desired inequality comes from the relations $\sum_{i=1}^{n}\left(d_{i}-\frac{2m}{n}\right)^{2}=\frac{\sigma_{t}(G)}{n}$ and $\sum_{i=1}^{n}\sqrt{d_{i}}\geq n\sqrt{\delta}$.
\end{proof}

Let $\lambda_{1}, \ldots, \lambda_{n}$ be the eigenvalues of $G$.
The energy of a graph, defined as $\mathcal{E}(G)=|\lambda_{1}|+\cdots+|\lambda_{n}|$, represents an important concept in mathematical chemistry. 
It is known that $\mathcal{E}(G)\leq \sqrt{2mn}$, see \cite{mc}, and $\mathcal{E}(G)\leq \sum_{i=1}^{n}\sqrt{d_{i}},$ see \cite{ariz}.
From (\ref{in}) we are in a position to obtain a new upper bound for the graph energy.

\begin{corollary}
\label{petaa}
Let $G$ be a connected graph on $n$ vertices, $m$ edges, with a minimum degree $\delta$ and a maximum degree $\Delta.$ Then
\begin{equation}
\label{energy}
\mathcal{E}(G)\leq \sqrt{2mn}-\frac{n\sqrt{\delta}\sigma_{t}(G)}{4(n^{2}\Delta^{2}+4m^{2})}.
\end{equation}
\end{corollary}

\begin{Remark} The bound in (\ref{energy}) represents an improvement of the well-known McClelland upper bound $\mathcal{E}(G)\leq \sqrt{2mn}$.
From this relation we observe that among all connected graphs with a fixed number of vertices and edges, and with a fixed minimum and maximum degree, the graphs with a larger $\sigma_{t}$-irregularity have a smaller energy, and vice versa.
\end{Remark}

\section{Irregular graphs with minimal $\sigma_{t}$-irregularity}

The first result in this section presents a lower bound for $\sigma_{t}$ in terms of $n, m$ and the largest degree $\Delta$.

\begin{theorem}
\label{a1}
Let $G$ be a connected graph on $n$ vertices and $m$ edges and let $k (<n)$ vertices be of maximum degree $\Delta$. Then
$$\sigma_{t}(G)\geq \frac{k}{n-k}\cdot (n\Delta-2m)^{2}.$$
The equality holds if and only if $G$ is a graph where $k$ vertices have maximum degree $\Delta$ and $n-k$ vertices have degree $\delta=\frac{2m-k\Delta}{n-k}.$
\end{theorem}

\begin{proof}
We begin by proving the inequality:
\begin{equation}\label{ner} k\Delta^{2}+d_{k+1}^{2}+\cdots+d_{n}^{2}\geq \frac{(k\Delta+d_{k+1}+\cdots+d_{n})^{2}}{n}+k\cdot\frac{((n-k)\Delta-(d_{k+1}+d_{k+2}+\cdots+d_{n}))^{2}}{n(n-k)}.
\end{equation}
The argument proceeds as follows.
If we let $S=d_{k+1}+d_{k+2}+\cdots+d_{n}$, the inequality (\ref{ner}) is equivalent to
$$n(n-k)(k\Delta ^{2}+d_{k+1}^{2}+\cdots+d_{n}^{2})\geq(n-k)(k\Delta+S)^{2}+k((n-k)\Delta-S)^{2}$$
which in turn is equivalent to
\begin{equation}
\label{inequ}
(n-k)(d_{k+1}^{2}+\cdots+d_{n}^{2})\geq S^{2}.
\end{equation}
Dividing (\ref{inequ}) by $n-k$ yields the equivalent inequality
$$d_{k+1}^{2}+d_{k+2}^{2}+\cdots+d_{n}^{2}\geq \frac{(d_{k+1}+d_{k+2}+\cdots+d_{n})^{2}}{n-k}$$
which holds because of the inequality between the quadratic and arithmetic means of the numbers $d_{k+1}, d_{k+2},\ldots,d_{n}$. Consequently, (\ref{ner}) holds as well.
Now, the inequality in (\ref{ner}) is equivalent to
$$M_{1}(G)=k\Delta^{2}+d_{k+1}^{2}+\cdots+d_{n}^{2}\geq \frac{4m^{2}}{n}+k\cdot \frac{(n\Delta-2m)^{2}}{n(n-k)}.$$
Hence 
\begin{equation}
\label{final}
\sigma_{t}(G)=nM_{1}(G)-4m^{2}\geq \frac{k}{n-k} \cdot (n\Delta-2m)^{2}.
\end{equation}
\end{proof}

As a consequence of the above theorem we get the following result.

\begin{corollary}
\label{a2}
Let $G$ be a connected graph on $n$ vertices and $m$ edges and let $\Delta$ be its maximum degree. 
Then
$$
\sigma_{t}(G)\geq \frac{(n\Delta-2m)^{2}}{n-1}.
$$
The equality holds if and only if $G$ is a graph with one vertex of maximum degree $\Delta$ and the remaining vertices of degree $\delta=\frac{2m-\Delta}{n-1}.$
\end{corollary}

\begin{proof}
Since $\frac{k}{n-k}=-1+\frac{n}{n-k}\geq -1+\frac{n}{n-1}=\frac{1}{n-1}$ we get that $\sigma_{t}(G)\geq \frac{(n\Delta-2m)^{2}}{n-1}.$ The equality occurs if and only if exactly one vertex of $G$ 
has maximum degree $\Delta$ and the remaining $n-1$ vertices have degree $\frac{2m-\Delta}{n-1}.$
\end{proof}

In \cite{ds} it was proved that among all trees on $n$ vertices, the paths have the smallest $\sigma_{t}$-index.
Based on the previous theorem, we give an alternative proof of this result.

\begin{corollary}\label{a3} For every tree $T$ on $n$ vertices, we have
$$
\sigma_{t}(T)\geq 2n-4.
$$
Moreover, the equality holds if and only if $T$ is the path $P_n$.
\end{corollary}

\begin{proof}
From $m=n-1$ and Theorem \ref{a1} we get 
\begin{equation}
\label{tree}
\sigma_{t}(T)\geq \frac{k}{n-k} \big(n(\Delta-2)+2\big)^{2}.
\end{equation}
If $\Delta=2$, then $T$ is $P_n$. 
In this case $k=n-2$ and $\sigma_{t}(T)=2n-4.$ Suppose now $\Delta=3$ and $k=1$.
Then the tree $T$ contains $3$ pendent vertices and $n-4$ vertices of degree $2$.
In this case we get
$$
\sigma_{t}(T)=(n-4)(3-2)^{2}+3(3-1)^{2}+3(n-4)(2-1)^{2}=4n-4>2n-4.
$$
If $\Delta = 3$ and $k\geq2$ , from (\ref{tree}) we have
$$
\sigma_{t}(T)\geq \frac{2}{n-2}\cdot (n+2)^{2}>2n-4.
$$
Finally, if $\Delta \geq 4$, then we get
$$
\sigma_{t}(T)\geq \frac{1}{n-1}\cdot (2n+2)^{2}>2n-4.
$$
\end{proof}

In the next result we give a shorter and simpler proof of Theorem 2.1 published in \cite{ds}.

\begin{theorem}\label{a4}
Let $G$ be a simple, undirected graph on $n$ vertices. If $G$ is not a regular graph, then
\begin{equation}
\sigma_{t}(G)\geq
\left\{
\begin{array}{ll} n-1, & \mbox{if $n$ is odd;} \\
2n-4, & \mbox{otherwise.} 
\end{array}
\right.
\end{equation}
\end{theorem}

\begin{proof}
Let $v_{1}$ and $v_{n}$ be vertices with the maximum and the minimum degrees, respectively. There exists a vertex $v_{i}$ that is connected to $v_{1}$ and that is not connected to $v_{n}$.
Based on Lemma 5, by removing the edge $v_{1}v_{i}$ and connecting $v_{i}$ and $v_{n}$ we decrease the $\sigma_{t}$-irregularity if $d(v_1)\ge d(v_n)+2$.
We repeat this procedure until possible.

We can finish in two sequences.
Either $d(v_1)=a+1$, $d(v_n)=a-1$ and $d(v_j)=a$ if $2\le j\le n-1$ (in which case the next graph would be regular), or $d(v_1)=a$, $d(v_n)=a-1$ and $a\ge d(v_j)\ge a-1$ if $2\le j\le n-1$ (in which 
case our operation does not work anymore).
However, in the first case $v_n$ cannot be connected to all vertices $v_2,\dots,v_{n-1}$, and connecting it to a new vertex decreases $\sigma_t$ and reduces this case to the second one.

Hence, the graph minimizing $\sigma_{t}$ contains $x$ vertices of degree $a$ and $n-x$ vertices of degree $a-1$.
Then
$$
\sigma_{t}(G)=x(n-x)\big(a-(a-1)\big)^{2}=x(n-x).
$$
Consequently, the minimum $\sigma_t$ is achieved if $x=1$ or $x=n-1$, and $\sigma_{t}(G)=n-1$.
However, the sum of the vertex degrees is an even number, and so
$xa+(n-x)(a-1)=n(a-1)+x$ is even.
Hence if $n$ is even, then $x$ must be even as well.
And since $x(n-x)$ should be as small as possible, we conclude that $x=2$ or $x=n-2$, and $\sigma_{t}(G)=2(n-2)$.
\end{proof}

\section{Sigma index versus sigma total index via Laplacian eigenvalues}

Let $G$ be a connected graph with vertex set $V(G)$ and edge set $E(G)$. Let us suppose that the vertices of $G$ are labelled $\{1, 2, 3, \ldots, n\}$ with corresponding degrees $\{d_{1}, d_{2}, 
d_{3}, \ldots, d_{n}\}.$  The Laplacian matrix $L$ of the graph $G$ is defined as
\begin{equation}
L_{ij}=
\left\{
\begin{array}{rl}
-1, & \mbox{if $ij\in E(G)$}, \\
0, & \mbox {if $ij\notin E(G)$ and $i\neq j$}, \\
d_{i}, & \mbox{if $i=j$}.
\end{array}
\right.
\end{equation}

It is obvious from the definition that $L$ is a positive semidefinite matrix. Surveys of its variegated properties can be
found in \cite{meris, mohar}. The quadratic form defined by $L$ has the following useful expression
(where we identify the vector $x\in \mathbb{R}$  with a function $x: V(G)\rightarrow \mathbb{R}$):
\begin{equation}
\label{laplacian}
x^{T}Lx=\sum_{uv\in E(G)}\big(x(u)-x(v)\big)^{2}.
\end{equation}
 Let $\mu_{1}\leq \mu_{2}\leq \cdots \leq \mu_{n}$ be the Laplacian eigenvalues of $G$.
 The results in this section follow from the well-known Fiedler's  characterization of the largest and the second smallest Laplacian eigenvalue of $G$, that is, $\mu_{n}$ and $\mu_{2},$ see 
 \cite{fiedler}.
 
\begin{proposition}
\label{b1}
Let $G$ be a connected graph on $n$ vertices.
Then
$$
\mu_{n}=n \max _{x} \frac{\sum_{uv\in E(G)}(x(u)-x(v))^{2}}{\sum_{\{u,v\}\subseteq V(G)}(x(u)-x(v))^{2}},
$$
and
$$
\mu_{2}=n \min _{x} \frac{\sum_{uv\in E(G)}(x(u)-x(v))^{2}}{\sum_{\{u,v\}\subseteq V(G)}(x(u)-x(v))^{2}},
$$
where $x$ is a nonconstant vector.
\end{proposition}

Proposition \ref{b1} implies the following inequality.

\begin{proposition}
\label{b2}
Let $G$ be a connected graph on $n$ vertices. Then
\begin{equation}
\label{relation} 
\sigma(G)\leq \frac{\mu_{n}}{n}\sigma_{t}(G),
\end{equation}
where $\mu_{n}$ is the largest Laplacian eigenvalue of $G$.
\end{proposition}

\begin{proof} Let us suppose that $G$ is an irregular graph. Then the vector $d=(d_{1}, \ldots, d_{n})$ is nonconstant.
From (\ref{laplacian}) we get $d^{T}Ld=\sum_{uv\in E(G)}\big(d(u)-d(v)\big)^{2}=\sigma(G)$.
Now applying Proposition \ref{b1} we get
$$
\mu_{n}\geq n \cdot \frac{\sum_{uv\in E(G)}(d(u)-d(v))^{2}}{\sum_{\{u,v\}\subseteq V(G)}(d(u)-d(v))^{2}}= n \cdot \frac{\sigma(G)}{\sigma_{t}(G)}=\frac{n\sigma(G)}{\sigma_{t}(G)}.
$$
If $G$ is a regular graph, then $\sigma(G)=\sigma_{t}(G)=0$, so inequality (\ref{relation}) trivially holds.
\end{proof}

Also the next result is based on Proposition \ref{b1} and follows using the same reasoning like in Proposition \ref{b2}.

\begin{proposition}
\label{b4}
Let $G$ be a connected graph on $n$ vertices. Then
$$
\sigma_{t}(G)\leq \frac{n}{\mu_{2}}\sigma(G),
$$
where $\mu_{2}$ is the second smallest Laplacian eigenvalue of $G$.
\end{proposition}

It is interesting to characterize graphs $G$ for which $\sigma_t(G)=\sigma(G)$.
A graph $G$ is called \emph{a generalized complete $k$-partite graph} if its vertices can be partitioned into $k$ nonempty sets $V_1,V_2,\dots,V_k$, such that subgraphs induced by $V_i$ are regular 
and if $u\in V_i$, $v\in V_j$, $1\le i<j\le k$, then $uv$ is an edge of $G$.
In other words, $G$ is obtained from a complete $k$-partite graph with partition $V_1,V_2,\dots,V_k$, where the partite sets are not necessarily independent, but they induce regular graphs.

\begin{proposition}\label{b3}
\label{prop:sigma=sigmat}
We have $\sigma(G)=\sigma_t(G)$ if and only if $G$ is a generalized complete $k$-partite graph.
\end{proposition}

\begin{proof}
Observe that
$$
\sigma_t(G)=\sigma(G)+\sum_{uv\notin E(G)}\big(d(u)-d(v)\big)^2.
$$
Hence $\sigma_t(G)=\sigma(G)$ if and only if $d(u)=d(v)$ whenever $uv\notin E(G)$.
So split the vertices of $G$ into sets according to their degree.
This gives a partition and $\sigma_t(G)=\sigma(G)$ if and only if any two vertices from different partite sets are connected.
The rest follows from the fact that each partite set contains vertices of the same degree.
\end{proof}

Observe that an irregular bipartite graph is a generalized complete $k$-partite graph if and only if it is a complete bipartite graph, while a tree is a generalized complete $k$-partite graph if and only if it 
is the star graph $S_n$.

%In order to determine the graphs for which the equality in (\ref{relation}) does hold, we recall the following  well-known results published in Kel'mans \cite{rusian}, Anderson et al. 
%\cite{anderson} and Gutman \cite{gutman}.

%\begin{itemize}
%\item (Kel'mans) If $G$ is a simple graph then $\mu_{n}\leq n$ with equality if and only if the complement of $G$ is not connected.

%\item (Anderson et al.) $\mu_{n}\leq \max \{d(u)+d(v): uv\in E(G)\}.$ If $G$ is connected then the equality holds if and only if $G$ is bipartite semiregular.

%\item (Gutman) Among all trees with a fixed number of vertices the star has the
%greatest value of the greatest Laplacian eigenvalue, it holds $\mu_{n}(S_{n})=n.$ Moreover, if $T_{n}$ is any $n$-vertex tree, not isomorphic to $S_{n}$, then $\mu_{n}(T_{n})<n.$

%\end{itemize}

%From Proposition 7 and by the above results we get the following results. 
%\begin{theorem} Let $G$ be a connected graph on $n$ vertices. Then
%\begin{equation}\label{equal}
%\sigma(G)\leq \sigma_{t}(G).
%\end{equation}
%\end{theorem}
%\begin{corollary}
%Among all connected graphs on $n$ vertices, the equality in (\ref{equal}) holds if and only if the complement of $G$ is not connected.
%\end{corollary}

%\begin{corollary} Among all connected bipartite graphs on $n$ vertices, the equality in (\ref{equal}) holds if and only if $G$ is a bipartite semiregular graph.
%\end{corollary}
% \begin{corollary}Among all trees on $n$ vertices, the equality in (\ref{equal}) holds if and only if $G$ is the star graph $S_{n}.$
% \end{corollary}

\section{Concluding remarks and further work}

As we already stated, $\sigma_{t}$ is an irregularity measure defined as a natural modification of the Albertson irregularity $\sigma$. 
In this paper we examined graphs with the maximum and the minimum $\sigma_{t}$-irregularity.
Our consideration addressed various types of graphs: general irregular graphs, triangle-free graphs, bipartite graphs and trees.
Related to the triangle-free graphs, in Proposition~{\ref{tretaa}} we proved that among complete bipartite graphs on $n$ vertices,  either the complete bipartite graph 
$K_{\lfloor\frac{n}{4}(2-\sqrt{2})\rfloor, \lceil\frac{n}{4}(2+\sqrt{2})\rceil }$ or $K_{\lceil\frac{n}{4}(2-\sqrt{2})\rceil, \lfloor\frac{n}{4}(2+\sqrt{2})\rfloor}$ has the maximum 
$\sigma_{t}$-irregularity.
We believe that among all triangle-free graphs the complete bipartite graphs have the maximum $\sigma_{t}(G)$.
Thus, we pose the following conjecture:

\begin{Conjecture}\label{c1}
Among connected triangle-free graphs on $n$ vertices, the extremal complete bipartite graphs have the maximum $\sigma_{t}(G)$-irregularity.
\end{Conjecture}

As a first step towards the proof of Conjecture~{\ref{c1}} one would expect to prove it for bipartite graphs.
By the definition of $\sigma$ and $\sigma_{t}$ it is obvious that $\sigma(G)\leq \sigma_{t}(G)$.
In order to characterize graphs for which $\sigma(G)=\sigma_{t}(G),$ we defined generalized complete $k$-partite graphs. In the class of trees, we proved that the star $S_{n}$ is the unique graph for 
which $\sigma(G)=\sigma_{t}(G).$ 
Motivated by the inequality $\sigma(G)\leq \sigma_{t}(G),$ it is natural to ask the opposite question: is there a relation of type $\sigma_{t}(G)\leq c\cdot \sigma(G)$, where $c$ is a constant or a 
function of $n$?
(Observe that Proposition~{\ref{b4}} is not of this type.)
For trees on $n$ vertices, we pose the following conjecture:

\begin{Conjecture}
Let $T$ be an $n$-vertex tree.
Then
\begin{equation}
\label{conj}
\sigma_{t}(T)\leq (n-2)\sigma(T).
\end{equation}
The equality holds if and only if $T$ is $P_{n}$.
\end{Conjecture}

%It is well known that $\sigma(G)=F(G)-2M_{2}(G),$ Thus, the relation (\ref{conj}) is equivalent to
%\begin{equation}
%\label{ekv}
%nM_{1}(T)-4(n-1)^{2}\leq (n-2)F(T)-2(n-2)M_{2}(T).
%\end{equation}
%In \cite{zhou} it was proved that $M_{1}(T)-M_{2}(T)\leq 2$, where the equality holds for paths on $n$ vertices, $P_{n}.$ Hence $$nM_{1}(T)-4(n-1)^{2}\leq 
%2n+nM_{2}(T)-4n^{2}+8n-4=nM_{2}(T)-4n^{2}+10n-4.$$
%Now, it suffices to show that $nM_{2}(T)-4n^{2}+10n-4\leq (n-2)F(T)-2(n-2)M_{2}(T)$, which is equivalent to
%\begin{equation}\label{fm2}
%(3n-4)M_{2}(T)\leq (n-2)F(T)+4n^{2}-10n+4.
%\end{equation}

%\begin{Observation} If the inequality $(3n-4)M_{2}(T)\leq (n-2)F(T)+4n^{2}-10n+4$ is true, then Conjecture 1 is valid.
%\end{Observation}

\textbf{Acknowledgement}
\vskip 2mm

S. Filipovski is supported in part by the Slovenian Research Agency (research program P1-0285 and research projects N1-0210, J1-3001, J1-3002, J1-3003, and J1-4414). D. Dimitrov has been supported by ARIS program P1-0383 and ARIS project J1-3002.
M. Knor has been partially supported by Slovak research grants
VEGA 1/0567/22, VEGA 1/0069/23, APVV-22-0005 and APVV-23-0076, the Slovenian Research Agency ARIS
program P1-0383 and ARIS project J1-3002.
R. \v{S}krekovski has been partially supported by the Slovenian Research Agency and Innovation (ARIS) program P1-0383, project J1-3002, and the annual work program of Rudolfovo.

\end{document}